\documentclass[preprint,12pt]{elsarticle} 
\usepackage[numbers]{natbib}

\usepackage[utf8]{inputenc}
\usepackage[english]{babel}
\usepackage{csquotes}
\usepackage{algorithm}
\usepackage[noend]{algpseudocode}
\algrenewcommand\algorithmicrequire{\textbf{Input:}}
\algrenewcommand\algorithmicensure{\textbf{Output:}}
\algnewcommand{\algorithmicand}{\textbf{ and }}
\algnewcommand{\algorithmicor}{\textbf{ or }}
\algnewcommand{\OR}{\algorithmicor}
\algnewcommand{\AND}{\algorithmicand}
\newcommand{\algorithmicbreak}{\textbf{break}}

\usepackage{amsmath}
\usepackage{amsfonts}
\usepackage{amssymb}
\usepackage{amsthm}
\usepackage{nicefrac}
\usepackage{algpseudocode,algorithm,algorithmicx}
\usepackage[short]{optidef}
\usepackage{mathrsfs}
\usepackage{colonequals}

\usepackage{tabularx}
\usepackage{makecell}
\usepackage{caption,subcaption}

\usepackage{pgfplots}
\usepackage{pgfplotstable}
\usepackage{tikz}
\usetikzlibrary{arrows, positioning, calc}
\usetikzlibrary{automata,snakes}
\usetikzlibrary{decorations}
\usetikzlibrary{decorations.markings,decorations.pathreplacing}
\usetikzlibrary{shapes.misc}
\usetikzlibrary{backgrounds, intersections}

\newtheorem{theorem}{Theorem}[section]

\newtheorem{corollary}[theorem]{Corollary}
\newtheorem{definition}[theorem]{Definition}
\newtheorem{remark}[theorem]{Remark}
\newtheorem{lemma}[theorem]{Lemma}

\newtheorem{example}[theorem]{Example}

\pgfplotsset{compat=1.14} 

\usepackage{amssymb}

\journal{}

\begin{document}

\begin{frontmatter}



\title{On the Universal Near Shortest Simple Paths Problem}

\renewcommand*{\thefootnote}{\fnsymbol{footnote}}
\author{Luca E. Schäfer\footnote{Corresponding author\\Email addresses: luca.schaefer@mathematik.uni-kl.de (Luca E. Schäfer), maier@mathematik.uni-kl.de (Andrea Maier), ruzika@mathematik.uni-kl.de (Stefan Ruzika)}, Andrea Maier, Stefan Ruzika}

\address{Department of Mathematics, Technische Universität Kaiserslautern, 67663 Kaiserslautern, Germany}

\begin{abstract}
This article generalizes the Near Shortest Paths Problem introduced by Byers and Waterman in 1984 using concepts of the Universal Shortest Path Problem  established by Turner and Hamacher in 2011. The generalization covers a variety of shortest path problems by introducing a universal weight vector. We apply this concept to the Near Shortest Paths Problem in a way that we are able to enumerate all universal near shortest simple paths.
We present two recursive algorithms to compute the set of universal near shortest simple paths between two prespecified vertices and evaluate the running time complexity per path enumerated with respect to different values of the universal weight vector. Further, we study the cardinality of a minimal complete set with respect to different values of the universal weight vector. 
\end{abstract}



\begin{keyword}
Near shortest Paths \sep Dynamic Programming \sep Universal Shortest Path \sep Complexity \sep Enumeration


\end{keyword}

\end{frontmatter}


\section{Introduction}
The \(K\)-shortest path problem is a well-studied generalization of the shortest path problem, where one aims to determine not only the shortest path, but also the \(K\) shortest paths ($K >1$)  between a source vertex \(s\) and a sink vertex \(t\). In contrast to the \(K\)-shortest path problem and its applications, (cf. \citep{aljazzar2011k,epp,christ,yen}), the near shortest paths problem, i.e., \textsf{NSPP}, received fewer attention in the literature, (cf. \citep{BW,CW,BW1}).
The \textsf{NSPP} aims to enumerate all paths whose lengths are within a factor of \(1 + \varepsilon\) of the shortest path length between source and sink for some \(\varepsilon > 0\). In \citep{BW}, a dynamic programming algorithm for enumerating all near shortest paths in a directed network is presented.
In \citep{CW}, \textsf{NSPP} is further elaborated and an algorithm for computing the set of near shortest simple paths is displayed. The authors apply \textsf{NSPP} to the \(K\)-shortest path problem and present computational results on grid and road networks.

On the contrary, the shortest path problem is one of the best known and most studied optimization problems in operations research with several applications, (cf. \citep{cher,conde2018minmax,gallo,taccari2016integer}). One can find several variations of shortest path problems such as the bottleneck, the \(k\)-max and the \(k\)-sum shortest path problem, etc., where one aims to minimize the largest, the \(k\)-th largest and the sum of the \(k\) largest arc costs among the set of feasible paths from source to sink, respectively, (cf. \citep{garfinkel2006thek,ruzika,bsp1}).

In \citep{turner2012variants} and \citep{turner2011universal}, the authors generalize the shortest path problem by introducing a universal weight vector \(\lambda\) to the objective function such that a variety of classical shortest path problems can be covered by this formulation. The problem is called the universal shortest path problem, i.e., \textsf{USPP}. A sequential definition, where one has to solve \(|V|-1\) subproblems with fixed cardinality, as well as a definition with cardinality \(|A|\) are proposed, where \(|V|\) and \(|A|\) refer to the number of vertices and arcs of a graph \(G=(V,A)\), respectively.
\newline

In this paper, we generalize \textsf{NSPP} to the universal near shortest simple paths problem, called \textsf{UNSSPP}. We show, that by using this formulation, the idea of \textsf{NSPP}, mentioned in \citep{BW}, can be applied to a variety of well-studied shortest path problems. In this context, we present two recursive algorithms to compute the set of universal near shortest simple paths. Our fastest algorithm (when applied to the classical near shortest paths problem) has the same running time complexity per path enumerated than the algorithm presented in \cite{CW} with the addition that it can be applied to almost any shortest path problem. Further, we prove that the amount of work per path enumerated is polynomially bounded, provided that the corresponding shortest path problem with respect to the value of the universal weight vector \(\lambda\) can be solved in polynomial time. Afterwards, we study the size of a minimal complete set with respect to different values of \(\lambda\). To the best of our knowledge, this problem has not been investigated in the literature so far.
\newline

The remainder of this article is structured as follows.
In Section \ref{sec:notation}, we state the notation used throughout this paper and explain some graph theoretical aspects. In Section \ref{sec:USPP}, we introduce the universal near shortest simple paths problem, called \textsf{UNSSPP}.
Section \ref{sec:IPFormulations} deals with the size of a minimal complete set of the universal near shortest simple paths problem with respect to different values of the universal weight vector \(\lambda\). 
Section \ref{sec:conclusion} summarizes and concludes the paper.

\section{Preliminaries and Notation}\label{sec:notation}
The algorithms presented in this paper rely on basics of graph theory and network optimization. Therefore, we briefly recall the most important definitions in this section.
Let \(G = (V,A)\) be a directed graph with vertex set $V$, arc set $A$ and let \(c: A \rightarrow \mathbb{Z}_+\) be a cost function over the set of arcs. We set the number of vertices \(|V|\) to \(n\) and the number of arcs \(|A|\) to \(m\). Further, we assume a source vertex \(s \in V\) and a target vertex \(t\neq s \in V\) to be given. A directed path \(P\)  from \(s\) to \(t\) is defined by a sequence of adjacent arcs in \(G\), i.e., \(P = (a_1,\ldots,a_k)\), where \(a_1\) is an outgoing arc of \(s\) and \(a_k\) is an ingoing arc of \(t\). Consequently, the total cost of a path \(P\) can be computed as the sum of all arc-costs on this path, i.e., \(c(P) \coloneqq \sum_{i = 1}^{k} c(a_i)\). If we aim to refer to the vertices and arcs of a path \(P\), we write \(V(P)\) and \(A(P)\), respectively. With \(l(P)\) we denote the length of path \(P\) with respect to the number of arcs on that path, i.e., \(l(P)\coloneqq |\{a \in A(P)\}|\). For the remainder of this article, we are interested in simple paths. A path \(P\) is called simple, if it contains no repeated vertices, i.e., \(l(P) \leq n-1\).  
With \(\mathcal{P}_{uv}\), we denote the set of all simple \(u\)-\(v\)-paths in \(G\).
Further, \(\delta_G^+(u)\) and \(\delta_G^-(u)\) represent the set of outgoing and incoming arcs of vertex \(u \in V\) in \(G\), respectively. If the underlying graph is clear from the context, we just write \(\delta^+(u)\) and \(\delta^-(u)\). For a path \(P\), we define \(\text{pred}_P(u)\) to be the predecessor vertex of \(u \in V(P)\) on path \(P\). Further, for some given arc set \(F \subseteq A\), we denote by \(G\setminus F\) the graph obtained from \(G\) by removing the arcs in \(F\).

Dijkstra's algorithm, cf. \citep{Ahuja}, can be used to compute the shortest path from the source vertex \(s\) to all other vertices in the graph \(G\). Using Fibonacci heaps, the algorithm has a worst case time complexity of \(\mathcal{O}(m + n \log n)\), see \citep{FIB}. The shortest path distance from \(v \in V\) to the target vertex \(t\) is denoted by \(d(v,t)\) for all \(v \in V\). Note that we have to solve only one single shortest path problem to obtain all shortest path distances for all \(v \in V\) to \(t\) by executing Dijkstra's algorithm on the inverse graph of \(G\), i.e., the graph obtained from \(G\) by reversing all arcs.

The near shortest paths problem aims to enumerate all those \(s\)-\(t\)-paths \(P\) with total cost smaller than a given bound \(B \in \mathbb{Z}_+\), i.e., \(c(P) \leq B\). Usually, \(B\) is set to be equal to \((1+\varepsilon)\cdot d(s,t)\) for some given \(\varepsilon > 0\) such that we can enumerate all paths from \(s\) to \(t\) whose total costs are within a factor of \(1+\varepsilon\) of the smallest path cost.

Algorithms for enumerating near shortest paths in general and near shortest simple paths are depicted in \citep{BW} and \citep{CW}, respectively.

In \citep{BW}, a dynamic programming algorithm for computing all near shortest paths between two vertices is proposed. The algorithm works as follows:

First, all values of \(d(v,t)\) are computed for all \(v \in V\). Then, a current \(s\)-\(v\)-path is extended via an arc \((v,w)\) if and only if \(\text{D} + c(v,w) + d(w,t) \leq B\), where D descibes the length of the current \(s\)-\(v\)-path. Whenever a path from \(s\) to \(t\) is found by applying the above described procedure, we output this path.

\section{Universal Near Shortest Simple Paths}\label{sec:USPP}
In \citep{turner2012variants} and \citep{turner2011universal}, the authors introduce the Universal Shortest Path Problem, i.e., \textsf{USPP} in two different variations, i.e., a sequential definition of \textsf{USPP} and a definition of \textsf{USPP} with cardinality \(|A|\). We focus on the latter definition, since the former definition requires to solve \(|V|-1\) universal subproblems with fixed cardinality, which cannot directly be applied to \textsf{NSPP}. Due to the fact that we restrict ourselves to simple paths, we consider \textsf{USPP} with cardinality \(|V|-1\), since a simple path uses at most \(|V|-1\) arcs.
The following definitions can be found in a slightly modified manner in \citep{turner2011universal}.
\begin{definition}[Extended sorted cost vector]
	Let \(P\) be a path in \(G\). Then, the extended sorted cost vector \(c_{\geq}(P) \in \mathbb{Z}_+^{n-1}\) for \(P\) is given by
	\begin{equation*}
	c_{\geq}(P) \coloneqq (c_{(1)}(P),\ldots,c_{(l(P))}(P),\underbrace{0,\ldots,0}_{n-1-l(P)})
	\end{equation*}
	where \(c_{(i)}(P), i = 1,\ldots,l(P)\) is the \(i\)-th largest arc cost in \(P\). Thus, \(c_{(1)}(P)\geq\dots\geq c_{(l(P))}(P) \geq 0\) and \(c_{(i)}(P) \coloneqq 0\) for \(i = l(P)+1,\ldots,n-1\) if \(l(P)< n-1\).
\end{definition}
\begin{definition}[Universal Shortest Path Problem]
	The Universal Shortest Path Problem in \(G\) with cardinality \(n-1\) and a universal weight vector \(\lambda \in \mathbb{Z}^{n-1}\), called \textsf{USPP(\(G,\lambda\))}, is defined as follows
	\begin{equation*}\tag{\textsf{USPP(\(G,\lambda\))}}
	\min_{P \in \mathcal{P}_{st}} f_{\lambda}(P) \coloneqq \sum_{i=1}^{n-1} \lambda_i c_{(i)}(P).
	\end{equation*}
	A path \(P^* \in \mathcal{P}_{st}\) that is optimal for \textsf{USPP(\(G,\lambda\))} is called universal shortest path with universal objective function value \(f_{\lambda}(P^*)\). Further, \(f^*_{\lambda}(u,v)\) denotes the universal optimal objective function value of the universal shortest \(u\)-\(v\)-path.
\end{definition}
It can be shown that \textsf{USPP(\(G,\lambda\))} is in general \(\mathcal{NP}\)-complete by setting \(\lambda_i = -1\) for \(i=1,\ldots,n-1\) and reducing the longest path problem, which is known to be \(\mathcal{NP}\)-complete, to \textsf{USPP(\(G,\lambda\))}, cf. \citep{turner2011universal}. 

Table \ref{tab:universal} reveals a selection of some well-studied objective functions that can be modeled by \textsf{USPP(\(G,\lambda\))}. Note that with the sequential definition of \textsf{USPP} a larger variety of shortest path problems can be modelled, cf. \citep{turner2011universal}. For example, the balanced shortest path problem (see e.g. \citep{martello,punnen,turner2011generalized}), where one aims to minimize the difference between the largest and smallest value of a feasible solution, cannot be modeled by \textsf{USPP(\(G,\lambda\))}. In the following, we assume that \(1 \leq k \leq n-1\). However, for the case of \(k=1\), the bottleneck shortest path problem coincides with the \(k\)-max and \(k\)-sum shortest path problem.
\begin{table}[htbp]
	\centering
	\begin{tabularx}{\textwidth}{p{0.25\linewidth}|p{0.15\linewidth}|X|X} \hline
		\(\lambda \in \mathbb{Z}^{n-1}\) & \(f_{\lambda}(P)\) & Optimization problem & Runtime\\ \hline 
		$(1,\ldots,1)$ & $\sum_{a \in P}c(a)$ & Sum Shortest Path Problem (SSP) & \(\mathcal{O}(m+n\log n)\), cf. \citep{FIB}\\
		$(1,0,\ldots,0)$ & $\max_{a \in P}c(a)$ & Bottleneck Shortest Path Problem (BSP) & \(\mathcal{O}(m\log\log m)\), cf. \citep{bsp1}\\
		$(\underbrace{1,\ldots,1}_k,0,\ldots,0)$ & $\sum_{i=1}^k c_{(i)}(P)$ & \(k\)-Sum Shortest Path Problem (kSSP) & \(\mathcal{O}(n^2m^2)\), cf. \citep{garfinkel2006thek}\\
		$(\underbrace{0,\ldots,0}_{k-1},1,0,\ldots,0)$ & $c_{(k)}(P)$ & \(k\)-Max Shortest Path Problem (kMSP) & \(\mathcal{O}((m+n\log n)\cdot \log n)\), cf. \citep{ruzika}\\
		\hline 
	\end{tabularx}
	\caption{Objective functions modeled by \textsf{(USPP(\(\lambda\)))}}
	\label{tab:universal}
\end{table}

To the best of our knowledge, the Universal Near Shortest Simple Paths Problem in \(G\), called 
\textsf{UNSSPP(\(G,\lambda\))}, has not been investigated in the literature. 
\textsf{UNSSPP(\(G,\lambda\))} aims to enumerate all simple \(s\)-\(t\)-paths \(P\), where the universal objective function value \(f_{\lambda}(P)\) is within a factor of \(1+\varepsilon\) of the universal optimal objective function value \(f^*_{\lambda}(s,t)\). 
Note that for \(\lambda = (1,\ldots,1)\), \textsf{UNSSPP(\(G,\lambda\))} reduces to the classical near shortest paths problem as depicted in \citep{CW}. 

In \citep{garfinkel2006thek}, the authors showed that the \(k\)-Sum Shortest Path Problem (kSSP), also called the \(k\)-Centrum Shortest Path Problem, i.e., $\lambda = (\underbrace{1,\ldots,1}_k,0,\ldots,0)$, is in general \(\mathcal{NP}\)-hard, which follows directly by a reduction from the shortest hamiltonian chain between two vertices. In (kSSP), one aims to minimize the sum of the \(k\) arcs with highest cost in any simple \(s\)-\(t\)-path.  They proposed a polynomial time algorithm for graphs with positive arc costs and \(k \leq \min\{l(P) : P \in \mathcal{P}_{st}\}\).
Further, a polynomial time extension for general \(k\) is discussed, if there does not exist a cycle of negative cost in the graph, cf. \citep{garfinkel2006thek}. In \citep{bsp1}, a \(\mathcal{O}(m\log\log m)\) algorithm to compute the bottleneck shortest path (minimize the arc with highest cost in any \(s\)-\(t\)-path) is presented. In \citep{ruzika}, the authors discuss a generalization of bottleneck optimization problems, where one aims to minimize the \(k\)-th largest cost coefficient among the feasible solutions of a combinatorial optimization problem. A bisection algorithm is displayed, which in case of the \(k\)-Max Shortest Path Problem runs in \(\mathcal{O}((m+n\log n)\cdot \log n)\).
\newline

Before considering \textsf{UNSSPP(\(G,\lambda\))}, we show that the complexity of finding the \enquote{next} universal shortest path with respect to a given path \(P'\) and its corresponding universal objective function value \(f_{\lambda}(P')\), depends on the given value of \(\lambda\).

Therefore, let \(\lambda \in \mathbb{Z}^{n-1}\) and let \(P' \in \mathcal{P}_{st}\) be an arbitrary \(s\)-\(t\)-path with \(f_{\lambda}(P')=\mu\). Consider the following optimization problem of finding the \enquote{next} universal shortest path with respect to \(P'\), called \textsf{NextUSP}.
\begin{mini*}
	{P \in \mathcal{P}_{st}}{f_{\lambda}(P)}{}{}	
	\addConstraint{f_{\lambda}(P)}{\geq \mu+1 = f_{\lambda}(P')+1}{}
\end{mini*}

The decision version of \textsf{NextUSP} is as follows:
\newline

\textbf{\textsf{NextUSP}\((G,\lambda,P',\psi)\)}. Given a directed graph \(G=(V,A)\), \(\lambda \in \mathbb{Z}^{n-1}\), two distinct vertices \(s, t \in V\), a path \(P' \in \mathcal{P}_{st}\) with \(\mu \coloneqq f_{\lambda}(P')\) and a value \(\psi \in \mathbb{Z}\) with \(\psi > \mu\). Decide whether there exists a path \(P \in \mathcal{P}_{st}\) such that \(f_{\lambda}(P) \geq \mu+1\) and \(f_{\lambda}(P)\leq \psi\).
\newline

In the following, we analyze the complexity of \textsf{NextUSP} with respect to different values of \(\lambda\). Note that Theorem \ref{thm:NextUSP} has recently been proven in a different context in \citep{goerigk2018ranking}.
\begin{theorem}\label{thm:NextUSP}
	For \(\lambda = (1,\ldots,1)\), the problem  \textsf{NextUSP} is \(\mathcal{NP}\)-complete.
\end{theorem}
\begin{proof}
	The decision version of the problem is clearly in \(\mathcal{NP}\), since given a path \(P \in \mathcal{P}_{st}\) one can check in polynomial time whether \(f_{\lambda}(P')+1 \leq f_{\lambda}(P) = c(P) \leq \psi\). To show that the decision version of this problem is \(\mathcal{NP}\)-complete, we conduct a polynomial time reduction from the longest path problem, which is known to be \(\mathcal{NP}\)-complete, cf. \citep{garey2002computers}.
	
	We reduce the longest path problem to \textsf{NextUSP} as follows.
	Given an instance \(G=(V,A),\ s,t \in V,\ c:A \rightarrow \mathbb{Z}_+ \text{ and }\ L \in \mathbb{Z}_+\) of the longest path problem (does there exist a path \(P\) from \(s\) to \(t\) with \(c(P)\geq L\)), we construct an instance of \textsf{NextUSP} as follows:
	Let \(G' = (V,A')\) with \(A' \coloneqq A \cup \{e\}\), where \(e = (s,t)\). Further, let \(c':A' \rightarrow \mathbb{Z}_+\), where \(c'(a) = c(a)\) for all \(a \in A\) and \(c'(e) = \mu := L-1\). The path \(P'\) solely consists of arc \(e\), i.e., \(P'=(e)\), and consequently, \(\mu=L-1\).
	Moreover, let \(\psi \coloneqq L + c_{max} \cdot|V|\) be an upper bound  on all possible path lengths, where \(c_{max} \coloneqq \max \{c(a):a\in A\}\).
	
	Now, there exists a path \(P\) in \(G'\) with \(\psi \geq f_{\lambda}(P) = \sum_{i=1}^{n-1}c'_{(i)}(P)\geq \mu+1\) if and only if there exists a path \(P\) in \(G\) with \(c(P) \geq L\). This reduction can be done in polynomial time, which concludes the proof.
\end{proof}

\begin{theorem}
	For \(\lambda = (-1,\ldots,-1)\), the problem \textsf{NextUSP} is \(\mathcal{NP}\)-complete.
\end{theorem}
\begin{proof}
	This follows immediately from the fact that the longest path problem is \(\mathcal{NP}\)-complete, cf. \citep{garey2002computers}.
\end{proof}

\begin{corollary}\label{thm:NextUSPI}
	For $\lambda = (\underbrace{1,\ldots,1}_k,0,\ldots,0)$ with \(k \in \mathbb{Z}_+ \text{ and } k\leq n-1\), the problem \textsf{NextUSP} is \(\mathcal{NP}\)-complete.
\end{corollary}
\begin{proof}
	Follows immediately from Theorem \ref{thm:NextUSP} for \(k := n-1\).
\end{proof}
Note that \(\mathcal{NP}\)-completeness also holds true for a generalized version of the \(k\)-sum objective.
\begin{corollary}
	For $\lambda = (\underbrace{0,\ldots,0}_j,\underbrace{1,\ldots,1}_k,\underbrace{0,\ldots,0}_l)$ with \(j, k, l \in \mathbb{Z}_+ \text{ and } j+k+l=n-1\), the problem \textsf{NextUSP} is \(\mathcal{NP}\)-complete.
\end{corollary}
\begin{proof}
	Again, this follows from Theorem \ref{thm:NextUSP} with \(j=l=0\) and \(k = n-1\).
\end{proof}

\begin{theorem}\label{thm:fixed}
	For $\lambda = (\underbrace{1,\ldots,1}_k,0,\ldots,0)$ and \(k \in \mathbb{Z}_+ \text{ with } k\leq n-1\) fixed, the problem \textsf{NextUSP} can be solved in polynomial time.
\end{theorem}
\begin{proof}
	The number of possibilities to draw \(k\) arcs out of \(A\) without replacement and without order is $m \choose k$, which is in \(\mathcal{O}(m^k)\) due to ${m \choose k} = \frac{m!}{k! \cdot (m-k)!} = \frac{m\cdot (m-1) \cdot \ldots \cdot (m-k+1)}{1\cdot \ldots \cdot k} \in \mathcal{O}(m^k)$ for \(k\leq n-1\) fixed and, thus, polynomial. Let \(\mathcal{R}\) be the collection of all those arc sets with \(k\) elements and let \(R^1 \in \mathcal{R}\) with \(f_{\lambda}(R^1) = \min\{f_{\lambda}(R) \in \mathcal{R}\mid f_{\lambda}(R) > \mu\} = \psi \in \mathbb{Z}_+\), where \(f_{\lambda}(R) \coloneqq \sum_{a\in R}c(a)\) and \(\mu = f_{\lambda}(P')\) for some given \(P' \in \mathcal{P}_{st}\).
	Further, let \(a^*\) be the arc in \(R^1\) with the smallest arc cost in \(R^1\), i.e., \(c(a^*) \leq c(a)\) for all \(a \in R^1\). Now, we delete all arcs in \(A\backslash R^1\) with cost greater than \(c(a^*)\). With \(Q\), we denote the remaining arcs. There are \(k!\) possibilities to sort the \(k\) arcs in \(R^1\). We aim to find a path from \(s\) to \(t\) using all arcs in \(R^1\) in the graph \(G\) with arc set \(Q\) with respect to one specified sorting. This can be done by \(k\) consecutive depth-first-search procedures, which can be done in \(\mathcal{O}(k \cdot (n+m))\). We test all of these \(k!\) possible sortings until we find an \(s\)-\(t\)-path. If no such feasible completion exists, we choose another collection \(R^2 \in \mathcal{R}\) with \(f_{\lambda}(R^2) = \min\{f_{\lambda}(R) \in \mathcal{R}\mid f_{\lambda}(R) \geq \psi = f_{\lambda}(R^1)\}\) and apply the same method as described above. We repeat this procedure at most \(m\choose k\) times. Thus, the total running time complexity of this procedure is in \(\mathcal{O}(k!\cdot k \cdot(n+m)\cdot {m\choose k})\), which concludes the proof.
\end{proof}

\begin{corollary}
	For $\lambda = (1,0,\ldots,0)$, the problem \textsf{NextUSP} can be solved in polynomial time.
\end{corollary}
\begin{proof}
	Follows from Theorem \ref{thm:fixed}.
\end{proof}

\begin{theorem}
	For $\lambda = (\underbrace{0,\ldots,0}_{k-1},1,0,\ldots,0)$ and \(k \in \mathbb{Z}_+ \text{ with } k\leq n-1\), the problem \textsf{NextUSP} is \(\mathcal{NP}\)-complete.
\end{theorem}
\begin{proof}
	The corresponding decision version of \textsf{NextUSP} is as follows: Given a path \(P'\) with \(f_{\lambda}(P') = c_{(k)}(P') = \mu\), an integer value  \(\psi \in \mathbb{Z}\) with \(\psi>\mu\), a graph \(G=(V,A)\), \(k \leq n-1\) with \(k \in \mathbb{Z}_+\) and two distinct vertices \(s,t \in V\), does there exist a path \(P \in \mathcal{P}_{st}\) such that \(c_{(k)}(P) \geq \mu+1\) and \(c_{(k)}(P)\leq \psi\).
	The decision version of the problem is clearly in \(\mathcal{NP}\). To show that the decision version of this problem is \(\mathcal{NP}\)-complete, we conduct a polynomial time reduction from the longest path problem, which is known to be \(\mathcal{NP}\)-complete, even in graphs with \(c(a)=1\) for all \(a \in A\), cf. \citep{garey2002computers}.

	Given an instance \(G=(V,A),\ s,t \in V,\ c: A\rightarrow \mathbb{Z}_+ \text{ with } c(a)=1 \text{ for all } a\in A \text{ and } L \in \mathbb{Z}_+\) with \(L\leq n-1\) of the longest path problem (does there exist a path \(P\) from \(s\) to \(t\) with \(c(P)\geq L\)), we construct an instance of \textsf{NextUSP} as follows:
	Let \(G' \coloneqq (V',A')\) with \(V' \coloneqq V \cup \{1,2,\ldots,L-1\}\) and \(A'~\coloneqq~A \cup \{(s,1), (L-1,t)\} \cup \{(i,i+1) \mid i \in \{1,\ldots, L-2\}\}\). Further, let \(\psi \coloneqq L + c_{max} \cdot|V|\), where \(c_{max} \coloneqq \max \{c(a) \mid a\in A\}\), \(k = L\) and consequently $\lambda' = (\underbrace{0,\ldots,0}_{k-1},1,0,\ldots,0)$. Moreover, we define a cost function \(c':A' \rightarrow \mathbb{Z}_+\) as follows:
	\begin{equation*}
	c'(a) \coloneqq \begin{cases}
	\mu + 1, \text{ if } a \in A\\
	\mu, \text{ if } a \in A'\backslash A
	\end{cases}
	\end{equation*}
	
	Let \(P' = (s,1,2,\ldots,L-1,t)\) be the path in \(G'\) with \(c_{(k)}(P') = \mu\).
	
	Then, there exists a path \(P\) in \(G'\) with \(\psi \geq c_{(k)}(P)\geq \mu + 1\) if and only if there exists a path \(P\) in \(G\) with \(c(P) \geq L\). Note that the bound \(L\) is polynomial in the input size since a simple longest path with unit costs can have at most length \(|V|-1\). Thus, it follows that this reduction can be done in polynomial time, which concludes the proof.
\end{proof}

We state two algorithms for the Universal Near Shortest Simple Paths Problem, i.e., \textsf{UNSSPP\((G,\lambda)\)}, see Algorithm~\ref{alg:UNSP} and Algorithm \ref{alg:UNSPimproved}. 

Algorithm \ref{alg:UNSP} works as follows. In the \(i\)-th iteration, let \(P\) be the current considered \(s\)-\(u\)-path for some \(u \in V\setminus\{t\}\). We extend \(P\) by an arc \(a = (u,v)\) if and only if the best possible completion of \(P\circ a \in \mathcal{P}_{sv}\) by a path \(P^* \in \mathcal{P}_{vt}\), is still smaller or equal to \(B\).
Further, to ensure simplicity of the path \(P\circ a \circ P^*\), we have to calculate a path \(P^*\) such that no vertex is repeated in \(P\circ a \circ P^*\), see Remark \ref{remark:1}. Every time we reach the sink vertex \(t\) applying the above described procedure, we output this path \(P\) along with the corresponding universal objective function value \(D_{\lambda}\).

\begin{remark}\label{remark:1}
	By temporarily deleting all outgoing arcs of vertices \(v \in V(P)\) of the current enumerated \(s\)-\(u\)-path except for the arcs in \(P\circ a\), we ensure that the universal shortest path problem solved in line 6 of Algorithm \ref{alg:UNSP} follows path \(P\). Consequently, \(P\circ a \circ P^*\) is the best possible completion of \(P\) via \(a\) to \(t\) with respect to \(\lambda\). Further, by deleting these arcs we ensure that \(P\circ a \circ P^*\) is simple.
\end{remark}

\begin{algorithm}
	\caption{Universal Near Shortest Simple Paths Algorithm}\label{alg:UNSP}
	\begin{algorithmic}[1]
		\Require{A digraph \(G=(V,A),\ c: A \rightarrow \mathbb{Z}_+\), source \(s\), sink \(t\), weight vector \(\lambda \in \mathbb{Z}^{n-1}\), the universal shortest path distance \(f^*_{\lambda}(s,t)\), \(B = (1+\varepsilon)\cdot f^*_{\lambda}(s,t),\ \varepsilon \geq 0\)}
		\Ensure{All universal near shortest simple paths from \(s\) to \(t\)}
		\newline
		
		\Function{\texttt{UNSSP}}{G,source,sink,P,$\text{D}_{\lambda}$}
		\If{source \(=\) sink}
		\State print(P, \(\text{D}_{\lambda}\))
		\Else
		\For{\(a=(\text{source},v) \in A\)}
		\If{\(\underset{P^* \in \mathcal{P}_{v,sink}}{\min} f_{\lambda}(P\circ a\circ P^*) \leq B \AND v \notin V(P)\)}
		\State put \(a\) on top of P
		\State \(\text{D}_{\lambda} \gets f_{\lambda}(P \circ a)\)
		\State \Call{\texttt{UNSSP}}{G,\(v\),sink,P,$\text{D}_{\lambda}$}
		\State pop \(a\) from P
		\EndIf
		\EndFor
		\EndIf
		\EndFunction
	\end{algorithmic}
	\vspace{0.5cm}
	\textbf{Initialization:} \texttt{UNSSP}(G,\(s\),\(t\),\([.]\),\(0\))
\end{algorithm}

\begin{theorem}
	Algorithm \ref{alg:UNSP} correctly computes the set of all universal near shortest simple paths between source \(s\) and sink \(t\).
\end{theorem}
\begin{proof}
	The statement follows immediately from line 6 of the algorithm, since we only extend a current \(s\)-\(u\)-path \(P\) via an arc \(a = (u,v)\), if \(P\) via \(a\) has possibly still a feasible completion, i.e., \(\underset{P^* \in \mathcal{P}_{v,sink}}{\min} f_{\lambda}(P\circ a\circ P^*) \leq (1+\varepsilon)~\cdot f^*_{\lambda}(s,t) \text{ and } v \notin V(P)\). By Remark \ref{remark:1}, we know that \(P\circ a \circ P^*\) contains only vertices not already present in the current enumerated \(s\)-\(u\)-path \(P\). Thus, it holds that \(P\) via \(a\) along the path \(P^*\) corresponding to the best possible completion of \(P\circ a\) is feasible.
\end{proof}

Further, due to the dynamic programming structure using a depth-first-search strategy of Algorithm \ref{alg:UNSP}, we can show that the amount of work per path enumerated is polynomially bounded, provided that the underlying universal shortest path problem can be solved in polynomial time, e.g., for the shortest path problems presented in Table \ref{tab:universal}.

\begin{theorem}\label{thm:poly}
	The amount of work per path enumerated in Algorithm \ref{alg:UNSP} for \(\lambda \in \mathbb{Z}_+^{n-1}\) is in \(\mathcal{O}(mT)\), where \(T\) denotes the time for solving \textsf{USPP\((G,\lambda)\)}.
\end{theorem}
\begin{proof}
	In line 6 of Algorithm~\ref{alg:UNSP}, we have to solve a universal shortest path problem using the technique mentioned in Remark \ref{remark:1} to ensure simplicity. 
	Further, the algorithm scans at most all of the arcs in \(G\) before generating the first path, which is in \(\mathcal{O}(mT)\), since every time we investigate an arc, we have to solve a universal shortest path problem. 
	Thus, the amount of work before generating a first path can assumed to be in \(\mathcal{O}(mT)\). 
	Further, we have to check whether the current enumerated path \(P\) extended by arc \(a\) is still a simple path. This can be done in \(\mathcal{O}(n)\) and can be neglected. 
	From now on, we have to start backtracking. Here, at most \(n-1\) steps are involved (in the worst case back to \(s\)), and there are at most \(n-1\) steps to extend the path back to \(t\), which is again in \(\mathcal{O}(mT)\) following the same argumentation as above. Since this argument holds true for all subsequent enumerated paths,  the amount of work per path enumerated is in \(\mathcal{O}(mT)\).	
\end{proof}

The second algorithm to solve \textsf{UNSSPP\((G,\lambda)\)}, see Algorithm \ref{alg:UNSPimproved}, is again a recursive algorithm with a better running time complexity per path enumerated than Algorithm \ref{alg:UNSP}. However, in practice, the running time of Algorithm~\ref{alg:UNSP} might be better compared to Algorithm~\ref{alg:UNSPimproved} as stated in Remark \ref{remark:alg}. The algorithm works as follows. In the initialization phase, the universal shortest path \(P\) with \(V(P)=(s=v_0,\ldots,v_k=t)\) is computed. Further, the graph \(G\) gets modified, i.e., \(G' = G\setminus F\), in such a way that the universal shortest \(s\)-\(t\)-path in \(G'\) has to follow \(P\) until its last arc \(a=(v_{k-1},v_k)\). One can interpret \(F\) as the set of forbidden arcs. We store arc \(a\) in \(C[v_{k-1}]\) until no more universal near shortest paths from \(s\) to \(t\) following \(P\) except for arc \(a\) exist. Note that \(C[v_{k-1}]\) can be seen as the set of outgoing arcs of \(v_{k-1}\), which have already been considered and all paths using these arcs have already been discovered. Certainly, the same procedure is recursively repeated for all universal near shortest paths \(P'\) following \(P\) except for arc \(a\). Afterwards, we remove \(a\) from \(C[v_{k-1}]\), i.e., \(C[v_{k-1}] = \emptyset\). We recursively repeat this procedure along \(P\) until we reach the start vertex \(s\).

\begin{algorithm}
	\caption{Universal Near Shortest Simple Paths Algorithm}\label{alg:UNSPimproved}
	\begin{algorithmic}[1]
		\Require{A digraph \(G=(V,A),\ c: A \rightarrow \mathbb{Z}_+\), source \(s\), sink \(t\), weight vector \(\lambda \in \mathbb{Z}^{n-1}\), the universal shortest path distance \(f^*_{\lambda}(s,t)\) from \(s\) to \(t\) in \(G\), \(B = (1+\varepsilon)\cdot f^*_{\lambda}(s,t),\ \varepsilon \geq 0\)}
		\Ensure{All universal near shortest simple paths from \(s\) to \(t\)}
		\newline
		
		\Function{\texttt{Initialization}}{}
		\State \(P \gets\) \textsf{USPP}\((G,\lambda)\) \Comment{\(V(P)=(s=v_0,v_1,\ldots,v_k=t)\)}
		\State \(F\gets \{\delta_G^+(v_i)\setminus\{(v_i,v_{i+1})\}\mid i = 0,1,\ldots,k-2\} \cup \{(v_{k-1},v_k)\}\)
		\State \(G' \gets G\setminus F\)
		\For{\(v \in V\)}
		\State \(C[v] \gets \emptyset\)
		\EndFor
		\State \(C[v_{k-1}] \gets C[v_{k-1}] \cup \{(v_{k-1},v_k)\}\)
		\State \(v \gets v_{k-1}\)
		\State \Return $G'$,$C$,$v$,$P$
		\EndFunction
		\newline
		
		\Function{\texttt{UNSSP}}{$G'$,$C$,$v$,$P$}
		\If{\(v = \emptyset\)}
		\State \Return \Comment{All universal near shortest paths found}
		\EndIf
		\State \(P' \gets\) \textsf{USPP}\((G',\lambda)\) \footnotesize{\Comment{\(V(P')=(s=v'_0,\ldots, v'_i=v,\ldots,v'_l=t)\) or \(V(P')=(.)\)}}\normalsize
		\If{\(P'=(.)\) \algorithmicor \(f_{\lambda}(P')>B\)}
		\State \(C[v] \gets \emptyset\)
		\State \(C[\text{pred}_P(v)]\gets C[\text{pred}_P(v)] \cup \{(\text{pred}_P(v),v)\}\) \Comment{\(\text{pred}_P(s)\coloneqq \emptyset\)}
		\State \(G'\gets G'\cup \delta_{G}^+(v) \cup \delta_{G}^+(\text{pred}_P(v))\setminus C[\text{pred}_P(v)]\)
		\State \(v \gets \text{pred}_P(v)\)
		\State \Call{\texttt{UNSSP}}{$G'$,$C$,$v$,$P$} 
		\Else
		\State print$(P')$
		\State \(C[v'_{l-1}] \gets C[v'_{l-1}]\cup \{(v'_{l-1},v'_l)\}\)
		\State \(F\gets \{\delta_G^+(v'_i)\setminus\{(v'_i,v'_{i+1})\}\mid i = 0,1,\ldots,l-2\} \cup \{(v'_{l-1},v'_l)\}\)
		\State \(G' \gets G'\setminus F\)
		\State \(v \gets v'_{l-1}\)
		\State \Call{\texttt{UNSSP}}{$G'$,$C$,$v$,$P'$}
		\EndIf
		\EndFunction
		\newline
		
		\Function{\texttt{main}}{}
		\State \(G',C,v,P \gets\) \Call{\texttt{Initialization}}{}
		\State \Call{\texttt{UNSSP}}{$G'$,$C$,$v$,$P$}
		\EndFunction
	\end{algorithmic}
\end{algorithm}

\begin{example}
	Figure \ref{fig:1} shows an example graph for which we aim to execute Algorithm \ref{alg:UNSPimproved} with respect to \(\lambda=(0,1,0,0)\) and \(B=3\). The execution steps of the algorithm are summarized in Table \ref{tab:iteration}, whereas the graph modification of every step is shown in Figure \ref{fig:images}. In the initialization phase, we compute the universal shortest path, which is highlighted in Figure \ref{fig:1} and we get \(F=\{(a,c),(b,t)\}\).
	In the first call of \texttt{UNSSP}, we aim again at computing the universal shortest path in the modified graph, which is shown in Figure \ref{fig:2}. Obviously, there is no \(s\)-\(t\)-path. Again, we modify the graph as it is shown in Figure \ref{fig:3}. This procedure is repeated until we get all universal near shortest paths from \(s\) to \(t\), namely \(P^1 = (s,a,b,t)\) and \(P^2 = (s,a,c,b,t)\) with universal objective function values of \(2\) and \(3\), respectively. Note that the paths \(P^3 = (s,a,c,t)\) with an objective value of \(4\), which is found in step \(4\), is not output since \(f_{\lambda}(P^3)>B=3\). Further, note that in step \(2\) another universal near shortest path is found and that we get \(F=\{(a,b),(c,t),(b,t)\}\). Lastly, in step 6, vertex \(v\) gets equal to \(\text{pred}_P(s)\coloneqq \emptyset\) such that the algorithm terminates in the next call.
	\begin{figure}[htb]
		\centering
		\begin{subfigure}{0.25\textwidth}
			\centering
			\begin{tikzpicture}[scale=0.7, every node/.style={scale=0.6}]
			\node[draw, circle,minimum size=0.1cm,inner sep=1pt] (s) at (0.5,10) {$s$};
			\node[draw, circle,minimum size=0.1cm,inner sep=1pt] (a) at (2,10) {$a$};
			\node[draw, circle,minimum size=0.1cm,inner sep=1pt] (b) at (3.5,11) {$b$};
			\node[draw, circle,minimum size=0.1cm,inner sep=1pt] (c) at (3.5,9) {$c$};
			\node[draw, circle,minimum size=0.1cm,inner sep=1pt] (t) at (5,10) {$t$};
			
			\draw[->,ultra thick] (s)--node[above]{$1$}(a);
			\draw [->, ultra thick] (a)--node[above]{$6$}(b);
			\draw [->,thick] (a)--node[below]{$4$}(c);
			\draw [->,thick] (c)--node[left]{$3$}(b);
			\draw [->, ultra thick] (b)--node[above]{$2$}(t);
			\draw [->,thick] (c)--node[below]{$5$}(t);		
			\end{tikzpicture}
			\caption{Initialization}
			\label{fig:1}
		\end{subfigure}\hfil 
		\begin{subfigure}{0.25\textwidth}
			\centering
			\begin{tikzpicture}[scale=0.7, every node/.style={scale=0.6}]
			\node[draw, circle,minimum size=0.1cm,inner sep=1pt] (s) at (0.5,10) {$s$};
			\node[draw, circle,minimum size=0.1cm,inner sep=1pt] (a) at (2,10) {$a$};
			\node[draw, circle,minimum size=0.1cm,inner sep=1pt] (b) at (3.5,11) {$b$};
			\node[draw, circle,minimum size=0.1cm,inner sep=1pt] (c) at (3.5,9) {$c$};
			\node[draw, circle,minimum size=0.1cm,inner sep=1pt] (t) at (5,10) {$t$};
			
			\draw[->,thick] (s)--node[above]{$1$}(a);
			\draw [->,thick] (a)--node[above]{$6$}(b);
			\draw [->,thick] (c)--node[left]{$3$}(b);
			\draw [->,thick] (c)--node[below]{$5$}(t);		
			\end{tikzpicture}
			\caption{Step 1}
			\label{fig:2}
		\end{subfigure}\hfil 
		\begin{subfigure}{0.25\textwidth}
			\centering
			\begin{tikzpicture}[scale=0.7, every node/.style={scale=0.6}]
			\node[draw, circle,minimum size=0.1cm,inner sep=1pt] (s) at (0.5,10) {$s$};
			\node[draw, circle,minimum size=0.1cm,inner sep=1pt] (a) at (2,10) {$a$};
			\node[draw, circle,minimum size=0.1cm,inner sep=1pt] (b) at (3.5,11) {$b$};
			\node[draw, circle,minimum size=0.1cm,inner sep=1pt] (c) at (3.5,9) {$c$};
			\node[draw, circle,minimum size=0.1cm,inner sep=1pt] (t) at (5,10) {$t$};
			
			\draw[->, ultra thick] (s)--node[above]{$1$}(a);
			\draw [->, ultra thick] (a)--node[below]{$4$}(c);
			\draw [->, ultra thick] (c)--node[left]{$3$}(b);
			\draw [->, ultra thick] (b)--node[above]{$2$}(t);
			\draw [->, thick] (c)--node[below]{$5$}(t);		
			\end{tikzpicture}
			\caption{Step 2}
			\label{fig:3}
		\end{subfigure}
		
		\medskip
		\begin{subfigure}{0.25\textwidth}
			\centering
			\begin{tikzpicture}[scale=0.7, every node/.style={scale=0.6}]
			\node[draw, circle,minimum size=0.1cm,inner sep=1pt] (s) at (0.5,10) {$s$};
			\node[draw, circle,minimum size=0.1cm,inner sep=1pt] (a) at (2,10) {$a$};
			\node[draw, circle,minimum size=0.1cm,inner sep=1pt] (b) at (3.5,11) {$b$};
			\node[draw, circle,minimum size=0.1cm,inner sep=1pt] (c) at (3.5,9) {$c$};
			\node[draw, circle,minimum size=0.1cm,inner sep=1pt] (t) at (5,10) {$t$};
			
			\draw[->,thick] (s)--node[above]{$1$}(a);
			\draw [->,thick] (a)--node[below]{$4$}(c);
			\draw [->,thick] (c)--node[left]{$3$}(b);	
			\end{tikzpicture}
			\caption{Step 3}
			\label{fig:4}
		\end{subfigure}\hfil 
		\begin{subfigure}{0.25\textwidth}
			\centering
			\begin{tikzpicture}[scale=0.7, every node/.style={scale=0.6}]
			\node[draw, circle,minimum size=0.1cm,inner sep=1pt] (s) at (0.5,10) {$s$};
			\node[draw, circle,minimum size=0.1cm,inner sep=1pt] (a) at (2,10) {$a$};
			\node[draw, circle,minimum size=0.1cm,inner sep=1pt] (b) at (3.5,11) {$b$};
			\node[draw, circle,minimum size=0.1cm,inner sep=1pt] (c) at (3.5,9) {$c$};
			\node[draw, circle,minimum size=0.1cm,inner sep=1pt] (t) at (5,10) {$t$};
			
			\draw[->, ultra thick] (s)--node[above]{$1$}(a);
			\draw [->, ultra thick] (a)--node[below]{$4$}(c);
			\draw [->,thick] (b)--node[above]{$2$}(t);
			\draw [->, ultra thick] (c)--node[below]{$5$}(t);		
			\end{tikzpicture}
			\caption{Step 4}
			\label{fig:5}
		\end{subfigure}\hfil 
		\begin{subfigure}{0.25\textwidth}
			\centering
			\begin{tikzpicture}[scale=0.7, every node/.style={scale=0.6}]
			\node[draw, circle,minimum size=0.1cm,inner sep=1pt] (s) at (0.5,10) {$s$};
			\node[draw, circle,minimum size=0.1cm,inner sep=1pt] (a) at (2,10) {$a$};
			\node[draw, circle,minimum size=0.1cm,inner sep=1pt] (b) at (3.5,11) {$b$};
			\node[draw, circle,minimum size=0.1cm,inner sep=1pt] (c) at (3.5,9) {$c$};
			\node[draw, circle,minimum size=0.1cm,inner sep=1pt] (t) at (5,10) {$t$};
			
			\draw[->,thick] (s)--node[above]{$1$}(a);
			\draw [->,thick] (c)--node[left]{$3$}(b);
			\draw [->,thick] (b)--node[above]{$2$}(t);
			\draw [->,thick] (c)--node[below]{$5$}(t);		
			\end{tikzpicture}
			\caption{Step 5}
			\label{fig:6}
		\end{subfigure}
		\caption{Graph modification throughout the execution of the algorithm}
		\label{fig:images}
	\end{figure}
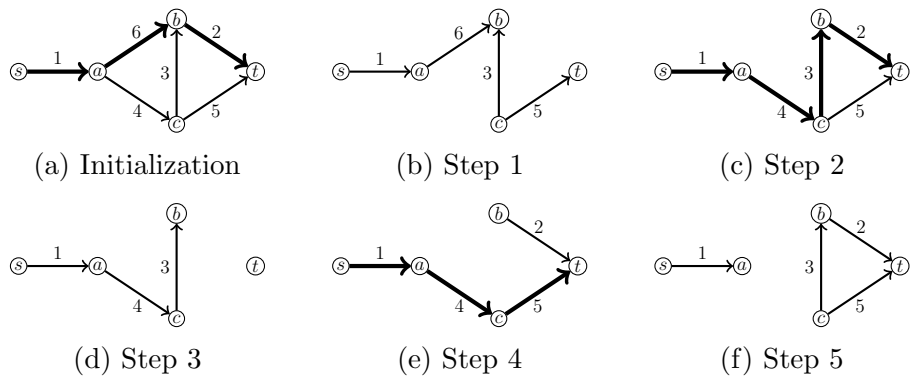
	
	\begin{table}[htbp]
		\centering
		\begin{tabular}{c|c|c|c|c|c|c|c|c} \hline
			Step & \(P\) & \(f_{\lambda}(P)\) &\(v\) & \(C[s]\) & \(C[a]\) & \(C[b]\) & \(C[c]\) & \(C[t]\)\\ \hline 
			Init & \((s,a,b,t)\) & \(2\) & \(b\) & - & - & \(\{(b,t)\}\) & - & -\\ \hline
			\(1\) & \((.)\) & - & \(a\) & - & \(\{(a,b)\}\) & - & - & - \\ \hline
			\(2\) & \((s,a,c,b,t)\) & \(3\) & \(b\) & - & \(\{(a,b)\}\) & \(\{(b,t)\}\) & - & - \\ \hline
			\(3\) & \((.)\)	& - & \(c\) & - & \(\{(a,b)\}\) & - & \(\{(c,b)\}\) & - \\	\hline
			\(4\) & \((s,a,c,t)\) & \(4\) & \(a\) & - & \(\{(a,b),\) & - & - & - \\ 
			&&&&&\((a,c)\}\)&&\\ \hline
			\(5\) & \((.)\) & - & \(s\) & \(\{(s,a)\}\) & - & - & - & - 	\\
			\hline 
		\end{tabular}
		\caption{Execution of Algorithm \ref{alg:UNSPimproved}}
		\label{tab:iteration}
	\end{table}
	
\end{example}

\begin{theorem}
	Algorithm \ref{alg:UNSPimproved} correctly computes the set of all universal near shortest simple paths between source \(s\) and sink \(t\).
\end{theorem}
\begin{proof}
	Due to construction, every path that is output by the algorithm is a universal near shortest path. We have to show that every path \(P\) with \(f_{\lambda}(P)\leq B\) is found by the algorithm. Therefore, assume that there exists a path \(P=(s=v_0,v_1,\ldots,v_k=t)\) with \(f_{\lambda}(P)\leq B\), which has not been found after termination of Algorithm \ref{alg:UNSPimproved}. Let \(\mathcal{P}^L := \{P^l \mid P^l = (s=v^l_0,\ldots,v^l_i,\ldots,v^l_{k(l)}=t), l = 1,\ldots,L\}\) be the set of all paths that have been found by the algorithm satisfying that \(v_j=v^l_j\) for all \(j=0,1,\ldots,i\) and \(i\) being maximal. Note that \(L\) refers to the number of paths in \(\mathcal{P}^L\). Without loss of generality we assume that the paths in \(\mathcal{P}^L\) are sorted in the order the algorithm prints them.
	Further, note that in the iteration when \(v=v^L_{i+1}\) and all paths \(P\in \mathcal{P}^L\) have been found, it follows that \((v_i,v^l_{i+1}) \in C[v_i]\) for all \(l=1,\ldots,L\). Furthermore, by lines 16 and 17 of the algorithm it follows that \(\delta^+_{G'}(v_j) = \{(v_j,v_{j+1})\}\) for all \(j = 0,1,\ldots,i-1\) and \(\delta^+_{G'}(v_i) = \delta^+_G(v_i)\setminus\{(v^l_i,v^l_{i+1})\mid l=1,\ldots,L\}\). Due to the algorithm, the subsequent recursion steps search for paths \(P'\) with \(P_{sv_i}\subset P'\), where \(P_{sv_i}\) refers to the \(s\)-\(v_i\)-subpath of \(P\). It follows that at one point there will be a path \(P^*\) with \((v_i,v_{i+1}) \in A(P^*)\), which is a contradiction to \(i\) being maximal.
\end{proof}

\begin{theorem}\label{thm:polyI}
	The amount of work per path enumerated in Algorithm \ref{alg:UNSPimproved} for \(\lambda \in \mathbb{Z}_+^{n-1}\) is in \(\mathcal{O}(nT)\), where \(T\) denotes the time for solving \textsf{USPP\((G,\lambda)\)}.
\end{theorem}
\begin{proof}
	After execution of \textsf{Initialization}, we get the universal shortest path (which is also a universal near shortest path) in time \(\mathcal{O}(m+T)\). Then, we start the backtracking procedure by recursively calling \textsf{UNSSP}.
	Since each universal near shortest simple path uses at most \(n\) vertices, \(\mathcal{O}(nT)\) work is involved for solving universal shortest path problems before the next path is generated (in the worst case, we have to go back to \(s\)). Besides solving universal shortest path problems, graph \(G\) and list \(C\) have to be modified, which is in \(\mathcal{O}(m)\). Since we have to do these modifications at most \(n\) times, the amount of work is in \(\mathcal{O}(nm)\), which can assumed to be in \(\mathcal{O}(nT)\) and can thus be neglected. Since this argument holds true for all subsequently enumerated paths,  the amount of work per path enumerated is in \(\mathcal{O}(nT)\).	
\end{proof}

\begin{remark}\label{remark:alg}
	By Theorems \ref{thm:poly} and \ref{thm:polyI}, we know that Algorithm \ref{alg:UNSPimproved} has a better running time complexity per path enumerated than Algorithm \ref{alg:UNSP}. Nevertheless, Algorithm \ref{alg:UNSP} has a few advantages over Algorithm \ref{alg:UNSPimproved}:
	\begin{itemize}
		\item it is easier to implement,
		\item it is better in terms of space complexity.
	\end{itemize}
\end{remark}

Next, we show that in case of \textsf{UNSSPP(\(G,\lambda\))}, we do not have to consider negative universal weight vectors, i.e., \(\lambda_i \leq 0\) for all \(i = 1,\ldots,n-1\). 
The following lemma can be found in \citep{turner2011universalPHD}.
\begin{lemma}
	Let \(\lambda_i \geq 0\) for all \(i = 1,\ldots,n-1\), let \(P, P' \in \mathcal{P}_{st}\) be two  \(s\)-\(t\)-paths and let \(c_{(i)}(P)\leq c_{(i)}(P')\) for all \(i=1,\ldots,n-1\). Then, it holds that
	\begin{equation*}
	f_{\lambda}(P)\leq f_{\lambda}(P').
	\end{equation*}
\end{lemma}
The following corollary follows immediately.
\begin{corollary}
	Let \(\lambda_i \leq 0\) for all \(i = 1,\ldots,n-1\), let \(P, P' \in \mathcal{P}_{st}\) be two  \(s\)-\(t\)-paths and let \(c_{(i)}(P)\leq c_{(i)}(P')\) for all \(i=1,\ldots,n-1\). Then, it holds that
	\begin{equation*}
	f_{\lambda}(P)\geq f_{\lambda}(P').
	\end{equation*}
\end{corollary}

\begin{theorem}
	Let \(\lambda_i \leq 0\) for all \(i = 1,\ldots,n-1\) with at least one strict inequality, let \(P^*\) be optimal for \textsf{USPP(\(\lambda\))} and let \(\varepsilon > 0\). Then, there is no path \(P \in \mathcal{P}_{st}\) with \(f_{\lambda}(P)\leq B = (1+\varepsilon)\cdot f_{\lambda}(P^*)\).
\end{theorem}
\begin{proof}
	Let \(P^*\) be optimal for \textsf{USPP(\(\lambda\))}. Then, it holds that \(f_{\lambda}(P^*) \leq f_{\lambda}(P) \text{ for all } P \in \mathcal{P}_{st}\). It follows that \(B = (1+\varepsilon)\cdot f_{\lambda}(P^*) < f_{\lambda}(P^*)\) due to \(\varepsilon > 0\) and \(\lambda_i \leq 0\) for all \(i = 1,\ldots,n-1\) with at least one strict inequality. Consequently, it holds that there is no path \(P \in \mathcal{P}_{st}\) with \(f_{\lambda}(P) < B = (1+\varepsilon)\cdot f_{\lambda}(P^*)\).
\end{proof}
Obviously, for the case of \(\varepsilon = 0\) and \(\lambda_i\leq 0\) for all \(i \in \{1,\ldots,n\}\) with at least one strict inequality, the path \(P^*\) optimal for \textsf{USPP(\(G,\lambda\))} fullfills the inequality.
Consequently, in the case of the universal near shortest simple paths problem, we focus on non-negative universal weight vectors \(\lambda \in \mathbb{Z}_+^{n-1}\).

\begin{theorem}
	\textsf{UNSSPP(\(G,\lambda\))} is intractable for any \(\lambda \in \mathbb{Z}^{n-1}\), i.e., the number of universal near shortest simple paths might be exponential in the number of vertices.
\end{theorem}
\begin{proof}
	We construct an instance, where the number of universal near shortest simple paths with respect to a given \(\varepsilon > 0\) is exponential in the number of vertices. Therefore, let \(G=(V,A)\) denote a directed graph with
	\begin{align*}
	&V=\{s=v_1,\ldots,v_n=t\}\\
	&A=\{(v_i,v_{i+1}), i = 1,4,7,\ldots,n-3\}\cup
	\{(v_i,v_{i+2}), i = 1,4,7,\ldots,n-3\}\cup\\
	&\{(v_i,v_{i+2}), i = 2,5,8,\ldots,n-2\}\cup
	\{(v_i,v_{i+1}), i = 3,6,9,\ldots,n-1\},
	\end{align*}
	where \(n-1\) is divisible by \(3\), i.e., \(n-1 \equiv 0 \mod 3\). There are \(\frac{4n-4}{3}\) arcs in the graph, if we construct the instance as described. 
	Let \(\lambda \in \mathbb{Z}^{n-1}\) be given and let \(f^*_{\lambda}(s,t) = \mathsf{OPT}\) be the universal optimal objective function value. By construction it holds that \(f^*_{\lambda}(P) = \mathsf{OPT}\) for all \(P \in \mathcal{P}_{st}\) and that \(l(P) = \frac{2n-2}{3}\) for all \(P \in \mathcal{P}_{st}\). Further, every \(s\)-\(t\)-paths is simple.
	\begin{figure}[h!]
		\centering
		\begin{tikzpicture}[scale=0.5]
		\node[draw, circle,inner sep=2pt,label={[left]\(v_1\)}] (s) at (0,10) {};
		\node[draw, circle,inner sep=2pt,label={[above]\(v_2\)}] (a) at (3,12) {};
		\node[draw, circle,inner sep=2pt,label={[below,yshift=-0.3cm]\(v_3\)}] (b) at (3,8) {};
		\node[draw, circle,inner sep=2pt,label={[above]\(v_4\)}] (c) at (6,10) {};
		\node[draw, circle,inner sep=2pt,label={[above]\(v_5\)}] (d) at (9,12) {};
		\node[draw, circle,inner sep=2pt,label={[below,yshift=-0.3cm]\(v_6\)}] (e) at (9,8) {};
		\node[draw, circle,inner sep=2pt,label={[right]\(v_7\)}] (f) at (12,10) {};
		\node[] (dots) at (14,10) {$ \dots $};
		\node[draw, circle,inner sep=2pt,label={[left]\(v_{n-3}\)}] (g) at (17,10) {};
		\node[draw, circle,inner sep=2pt,label={[above]\(v_{n-2}\)}] (h) at (20,12) {};
		\node[draw, circle,inner sep=2pt,label={[below,yshift=-0.3cm]\(v_{n-1}\)}] (i) at (20,8) {};
		\node[draw, circle,inner sep=2pt,label={[right]\(v_n\)}] (t) at (23,10) {};
		
		\draw[->,thick] (s)--node[above]{\(1\)}(a);
		\draw [->,thick] (a)--node[above]{\(1\)}(c);
		\draw [->,thick] (s)--node[below]{\(1\)}(b);
		\draw [->,thick] (b)--node[below]{\(1\)}(c);
		\draw [->,thick] (c)--node[above]{\(1\)}(d);
		\draw [->,thick] (c)--node[below]{\(1\)}(e);
		\draw [->,thick] (d)--node[above]{\(1\)}(f);
		\draw [->,thick] (e)--node[below]{\(1\)}(f);
		\draw [->,thick] (g)--node[above]{\(1\)}(h);
		\draw [->,thick] (g)--node[below]{\(1\)}(i);	
		\draw [->,thick] (h)--node[above]{\(1\)}(t);
		\draw [->,thick] (i)--node[below]{\(1\)}(t);		
		\end{tikzpicture}
		\caption{Intractability of the Universal Near Shortest Simple Paths Problem}
		\label{fig:intractable}
	\end{figure}
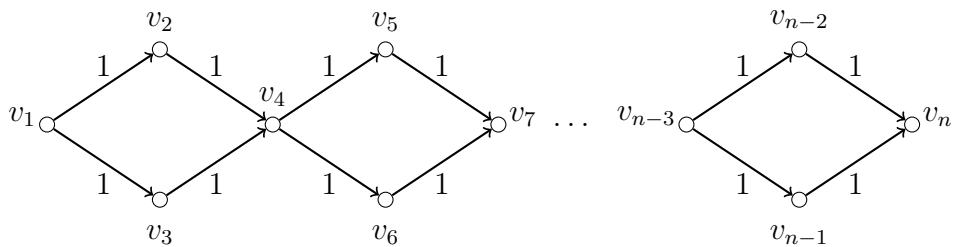
	It is easy to see, that for all \(\varepsilon \geq 0\), we have to enumerate all paths such that we get \(2^{\frac{n-1}{3}}\) universal near shortest paths from \(s\) to \(t\), which concludes the proof.
\end{proof}

\section{Cardinality of a Minimal Complete Set}\label{sec:IPFormulations}
The classical near shortest simple paths problem, i.e., \(\lambda = (1,\ldots,1)\), can iteratively be solved as an integer linear program, see optimization problem~\ref{opt:1}. Constraints (\ref{constra}) denote the flow conservation constraints, which together with the binary constraints (\ref{constrd}) ensure that a feasible solution \(x\) is an \(s\)-\(t\)-path. Further, constraints (\ref{constrc}) guarantee that an \(s\)-\(t\)-path is indeed a near shortest simple path, where we set initially \(\xi = f_{\lambda}^*(s,t) = d(s,t)\) to the universal optimal objective function value from \(s\) to \(t\). 

Throughout the algorithm, see Algorithm \ref{alg:NSPIPxi}, \(\xi\) is updated depending on the value of the near shortest path found in the previous iteration.  The updating procedure is repeated until \(\xi\) is equal to \(\lfloor B\rfloor + 1\), since any path with value greater than \(\lfloor B\rfloor\) is no near shortest path.

 We call this problem \textsf{NSPIP(\(\xi\))}.
\begin{mini!}
	{}{\sum_{(u,v) \in A}c_{uv}\cdot x_{uv}}{\label{opt:1}}{\label{obj1}}	
	\addConstraint{\sum_{(u,v) \in A} x_{uv} - \sum_{(v,u) \in A}x_{vu}}{= \begin{cases}
			1, \text{ if } u=s\\
			0, \text{ if } u\neq s,t\\
			-1, \text{ if } u=t
	\end{cases}}{\label{constra}}
	\addConstraint{\sum_{(u,v) \in A} c_{uv}\cdot x_{uv}}{\geq  \xi}{\label{constrc}}
	\addConstraint{x_{uv}}{\in \{0,1\} }{\forall (u,v) \in A \label{constrd}}
\end{mini!}
With \textsf{OPT\((\xi)\)}, we denote the optimal objective function value of \textsf{NSPIP(\(\xi\))}.

\begin{definition}[Minimal complete set]
	Let \(P_1, P_2 \in \mathcal{P}_{st}\), let \(\lambda \in \mathbb{Z}^{n-1}\) and let \(B = (1+\varepsilon)\cdot f^*_{\lambda}(s,t)\) for some \(\varepsilon > 0\). Further, let \(P^{NSP}(s,t) \coloneqq \{P \in \mathcal{P}_{st} \mid f_{\lambda}(P) \leq B\}\) be the set of all universal near shortest simple paths. We say that \(P_1\) is equivalent to \(P_2\) if and only if \(f_{\lambda}(P_1) = f_{\lambda}(P_2)\). A minimal complete set \(P^{NSP}_{min}(s,t) \subseteq P^{NSP}(s,t)\) is a set of universal near shortest simple paths such that all \(P \in P^{NSP}(s,t)\backslash P^{NSP}_{min}(s,t)\) are equivalent to exactly one \(P' \in P^{NSP}_{min}(s,t)\).
\end{definition}
By iteratively solving \textsf{NSPIP(\(\xi\))}, we get at most \((\lfloor B-d(s,t)\rfloor)\)-many universal near shortest simple paths for  \(\lambda = (1,\ldots,1)\), where  \((\lfloor B-d(s,t)\rfloor)~=~\lfloor \varepsilon\cdot d(s,t)\rfloor\). These paths denote a minimal complete set \(P^{NSP}_{min}(s,t)\). 

Algorithm \ref{alg:NSPIPxi} shows the procedure to compute a minimal complete set \(P^{NSP}_{min}(s,t)\). 
\begin{algorithm}
	\caption{Near Shortest Simple Paths Algorithm -- Minimal complete set}\label{alg:NSPIPxi}
	\begin{algorithmic}[1]
		\Require{A digraph \(G=(V,A),\ c: A \rightarrow \mathbb{Z}_+\), source \(s\), sink \(t\), the universal optimal objective function value from \(s\) to \(t\), i.e.,  \(\textsf{OPT\((0)\)} = d(s,t)\), \(B = (1+\varepsilon)\cdot d(s,t),\ \varepsilon > 0\)}
		\Ensure{A minimal complete set \(P^{NSP}_{min}(s,t)\)}
		\newline
		\State \(P^{NSP}_{min}(s,t) = \{\textsf{OPT(\(0\))}\}\)
		\State \(\xi \gets \textsf{OPT}(0)+1\)
		\While{\(\xi \neq \lfloor B\rfloor + 1\)}
		\State Solve \textsf{(NSPIP(\(\xi\)))}
		\If{\(\textsf{OPT\((\xi)\)}\leq B\)}
		\State \(P^{NSP}_{min}(s,t) \gets P^{NSP}_{min}(s,t) \cup \{\textsf{OPT\((\xi)\)}\}\)
		\State \(\xi \gets \textsf{OPT}(\xi)+1\)
		\Else \State \algorithmicbreak
		\EndIf
		\EndWhile\\
		
		\Return \(P^{NSP}_{min}(s,t)\)
	\end{algorithmic}
\end{algorithm}

In contrast to the classical near shortest simple paths problem, the universal near shortest simple paths problem, i.e., \textsf{UNSSPP(\(\lambda\))}, for an arbitrary \(\lambda \in \mathbb{Z}^{n-1}\) cannot directly be solved as an integer linear program, since constraints (\ref{constr6}) and objective function (\ref{obj2}) are non-linear, cf. \citep{turner2011universal}.
For general \(\lambda \in \mathbb{Z}^{n-1}\), binary sorting variables \(s_{i,uv}\) with
\begin{equation*}
s_{i,uv} = \begin{cases}
1, \footnotesize{\text{if } (u,v) \text{ is at position i of the corresponding extended sorted cost vector}}\\
0, \footnotesize{\text{else}}
\end{cases}
\end{equation*}
have to be introduced, which
together with constraints (\ref{constr5}) ensure that the arc costs are sorted correctly. Constraints (\ref{constr1}) and (\ref{constr7}) coincide with the respective constraints of \textsf{NSPIP(\(\xi\))}, whereas constraints (\ref{constr2}) denote the subtour elimination constraints. We call this problem \textsf{UNSPIP(\(\xi\))}. 

\begin{maxi!}
	{}{\sum_{i=1}^{n-1}\lambda_i\sum_{(u,v) \in A}s_{i,uv}\cdot c_{uv}\cdot x_{uv}}{}{\label{obj2}}	
	\addConstraint{\sum_{(u,v) \in A} x_{uv} - \sum_{(v,u) \in A}x_{vu}}{= \begin{cases}
			1, \text{ if } u=s\\
			0, \text{ if } u\neq s,t\\
			-1, \text{ if } u=t
	\end{cases}}{\label{constr1}}
	\addConstraint{\sum_{u \in S}\sum_{v \in S}x_{uv}}{\leq |S|-1}{\forall S \subseteq V, |S|\geq 2 \label{constr2}}
	\addConstraint{\sum_{i=1}^{n-1}s_{i,uv}}{=1}{\forall (u,v) \in A \label{constr3}}
	\addConstraint{\sum_{(u,v)\in A}s_{i,uv}}{=1}{\forall i=1,\ldots,n-1 \label{constr4}}
	\addConstraint{\sum_{(u,v) \in A}s_{i,uv}\cdot c_{uv}\cdot x_{uv}}{\geq \sum_{(u,v) \in A}s_{i+1,uv}\cdot c_{uv}\cdot x_{uv}\quad}{\forall i=1,\ldots,n-1 \label{constr5}}
	\addConstraint{\sum_{i=1}^{n-1}\lambda_i\sum_{(u,v) \in A}s_{i,uv}\cdot c_{uv}\cdot x_{uv}}{\leq f^*_{\lambda}(s,t) + \xi}{\label{constr6}}
	\addConstraint{s_{i,uv}}{\in \{0,1 \} }{\forall i=1,\ldots,n-1, (u,v)\in A}
	\addConstraint{x_{uv}}{\in \{0,1\} }{\forall (u,v) \in A\label{constr7}}
\end{maxi!}

We can linearize the problem by replacing constraints \ref{constr6} with the following:
\begin{align*}
&y_{i,uv} \leq s_{i,uv}  &\forall i = 1,\ldots,n-1, (u,v) \in A\\
&y_{i,uv} \leq x_{uv}  &\forall i = 1,\ldots,n-1, (u,v) \in A\\
&s_{i,uv} + x_{uv} - 1 \leq y_{i,uv}  &\forall i = 1,\ldots,n-1, (u,v) \in A\\
&\sum_{i = 1}^{n-1}\lambda_i\sum_{(u,v) \in A}c_{uv}\cdot y_{i,uv} \leq f^*_{\lambda}(s,t)+\xi\\
&y_{i,uv} \in \{0,1\}  &\forall i = 1,\ldots,n-1, (u,v) \in A
\end{align*}
which ensures \(y_{i,uv}\coloneqq s_{i,uv}\cdot x_{uv}\).

Objective function (\ref{obj2}) is replaced by
\begin{align*}
\sum_{i=1}^{n-1}\lambda_i\sum_{(u,v) \in A}c_{uv}\cdot y_{i,uv}.
\end{align*}

With \textsf{UOPT\((\xi)\)}, we denote the optimal objective function value of \textsf{UNSPIP(\(\xi\))} for \(\xi \in \{1,\ldots,\lfloor B-f^*_{\lambda}(s,t)\rfloor\}\).
Again, by solving this problem for all \(\xi = 1,\ldots,\lfloor B-f^*_{\lambda}(s,t)\rfloor\), we get at most \((\lfloor B-f^*_{\lambda}(s,t)\rfloor)\)-many universal near shortest simple paths for  \(\lambda \in \mathbb{Z}^{n-1}\). As described above, these paths denote a minimal complete set \(P^{NSP}_{min}(s,t)\). This can be done analogously to Algorithm~\ref{alg:NSPIPxi}.

\begin{remark}
	Note that in case of the universal near shortest simple paths problem, we need the subtour elimination constraints (\ref{constr2}) for the case of negative values for \(\lambda\). Again, if we assume \(\lambda_i\) to be positive for all \(i~=~1,\ldots,n-1\), \textsf{UNSPIP(\(\xi\))} can be formulated in an easier manner similar to optimization problem~\ref{opt:1}.
\end{remark}

Next, we investigate the cardinality of a minimal complete set with respect to \(\lambda \in \mathbb{Z}^{n-1}\). First, we show the most trivial case.
\begin{theorem}
	Let \(\lambda_i \leq 0\) for all \(i = 1,\ldots,n-1\) with at least one strict inequality and let \(B = (1+\varepsilon)\cdot f^*_{\lambda}(s,t)\) for some \(\varepsilon > 0\). Then, the cardinality of the minimal complete set \(P^{NSP}_{min}(s,t)\) of \textsf{UNSSPP(\(G,\lambda\))} is zero, i.e., \(|P^{NSP}_{min}(s,t)|=0\).
\end{theorem}
\begin{proof}
	Assume \(|P^{NSP}_{min}(s,t)|>0\). Then, there exists a path \(P' \in \mathcal{P}_{st}\) with \(f_{\lambda}(P') < f^*_{\lambda}(s,t)\), which is a contradiction to \(f^*_{\lambda}(s,t)\) being optimal for \textsf{USPP(\(G,\lambda\))}.
\end{proof}
Consequently, we investigate the case of \(\lambda_i\geq 0\) for all \(i \in\{1,\ldots,n-1\}\).

\begin{theorem}
	Let \(\lambda_i > 0\) for some \(i \in \{1,\ldots,n-1\}\) and \(\lambda_j = 0\) for all \(j \in \{1,\ldots,n-1\}, j\neq i\) and let \(B = (1+\varepsilon)\cdot f^*_{\lambda}(s,t)\) for some \(\varepsilon > 0\). Then, the cardinality of the minimal complete set \(P^{NSP}_{min}(s,t)\) of \textsf{UNSSPP(\(G,\lambda\))} is smaller or equal than the number of arcs in \(G\), i.e., \(|P^{NSP}_{min}(s,t)|\leq |A| = m\).
\end{theorem}
\begin{proof}
	This follows immediately from the fact that there are at most \(|A| = m\) many distinct arc cost in \(G\), i.e., \(c(a_1)\neq c(a_2)\) for all \(a_1,a_2 \in A, a_1\neq a_2\) and, thus, there are at most \(m\) many unique universal objective function values. 
\end{proof}

\begin{theorem}\label{thm:intractable}
	Let \(\lambda = (1,\ldots,1)\) and let \(B = (1+\varepsilon)\cdot f^*_{\lambda}(s,t)\) for some \(\varepsilon > 0\). Then, the cardinality of the minimal complete set \(P^{NSP}_{min}(s,t)\) of \textsf{UNSSPP(\(G,\lambda\))} is intractable, i.e., the cardinality of the minimal complete set might be exponential in the number of vertices.
\end{theorem}
\begin{proof}
	We construct an instance, where the cardinality of the minimal complete set with respect to a given \(\varepsilon > 0\) is exponential in the number of vertices, i.e., there are exponentially many (with respect to \(|V|\)) universal near shortest simple paths with distinct universal objective function values. Therefore, let \(G=(V,A)\) denote a directed graph with
	\begin{align*}
		& V=\{s=v_1,\ldots,v_n=t\},\\
		& A=\{(v_i,v_{i+1}), i = 1,\ldots,n-1\}\cup\{(v_i,v_{i+1})', i = 1,\ldots,n-1\}.
	\end{align*}
	 Let \(\lambda=(1,\ldots,1)\) and let \(\varepsilon = 2^{n-2}-1\), see Figure \ref{fig:intractable1}.
	
	\begin{figure}[h!]
		\centering
		\begin{tikzpicture}[scale=0.5]
		\node[draw, circle,inner sep=2pt,label={[left]\(v_1\)}] (s) at (0,10) {};
		\node[draw, circle,inner sep=2pt,label={[above,yshift=4pt]\(v_2\)}] (a) at (5,10) {};
		\node[draw, circle,inner sep=2pt,label={[above,yshift=4pt]\(v_3\)}] (b) at (10,10) {};
		\node[draw, circle,inner sep=2pt,label={[above,yshift=4pt]\(v_4\)}] (c) at (15,10) {};
		\node[] (dots) at (17,10) {$ \dots $};
		\node[draw, circle,inner sep=2pt,label={[above,yshift=4pt]\(v_{n-1}\)}] (d) at (19,10) {};
		\node[draw, circle,inner sep=2pt,label={[right]\(v_n\)}] (t) at (24,10) {};
		
		\draw [->,thick] (s)edge [bend angle=50, bend right] node[below]{\(1\)}(a);
		\draw [->,thick] (s)edge [bend angle=50, bend left] node[above]{\(1\)}(a);
		\draw [->,thick] (a)edge [bend angle=50, bend right] node[below]{\(0\)}(b);
		\draw [->,thick] (a)edge [bend angle=50, bend left] node[above]{\(2^0\)}(b);
		\draw [->,thick] (b)edge [bend angle=50, bend right] node[below]{\(0\)}(c);
		\draw [->,thick] (b)edge [bend angle=50, bend left] node[above]{\(2^1\)}(c);
		\draw [->,thick] (d)edge [bend angle=50, bend right] node[below]{\(0\)}(t);
		\draw [->,thick] (d)edge [bend angle=50, bend left] node[above]{\(2^{n-3}\)}(t);
		
		\end{tikzpicture}
		\caption{Intractability of the minimal complete set for \(\lambda=(1,\ldots,1)\)}
		\label{fig:intractable1}
	\end{figure}
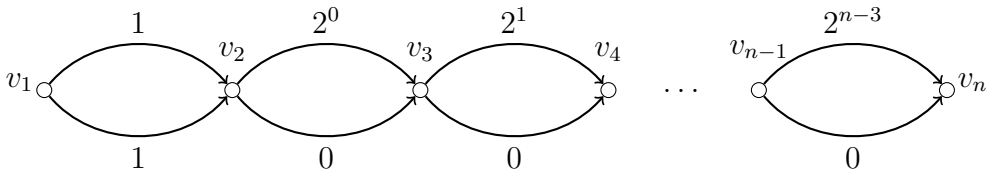
	One can see that the universal shortest path \(P\) follows the lower arcs of \(G\) with universal objective function value \(f_{\lambda}(P) = 1\), whereas the universal longest path \(P'\) follows the upper arcs of \(G\) with universal objective function \(f_{\lambda}(P') = 1+ \sum_{i=0}^{n-3} 2^i = 2^{n-2}\). By construction, it holds that for each \(z \in \{1,\ldots,2^{n-2}\}\) there are two distinct \(v_1\)-\(v_n\)-paths \(P_1\) and \(P_2\) with \(f_{\lambda}(P_i) = z\) for \(i=1,2\). Note that there are \(2^{n-1}\) different \(v_1\)-\(v_n\)-paths. Consequently, there are \(\frac{2^{n-1}}{2} = 2^{n-2}\) paths from \(v_1\) to \(v_n\) with distinct universal objective function values. Thus, it holds that \(|P^{NSP}_{min}(s,t)|=2^{n-2}\), which concludes the proof.
\end{proof}

\begin{corollary}\label{cor:ksum}
	Let $\lambda = (\underbrace{1,\ldots,1}_k,0,\ldots,0)$ with \(k\leq n-1\) and let \(B = (1+\varepsilon)\cdot f^*_{\lambda}(s,t)\) for some \(\varepsilon > 0\). Then, the cardinality of the minimal complete set \(P^{NSP}_{min}(s,t)\) of \textsf{UNSSPP(\(G,\lambda\))} is intractable, i.e., the cardinality of the minimal complete set might be exponential in the number of vertices.
\end{corollary}
\begin{proof}
	Follows immediately from Theorem \ref{thm:intractable} for \(k:=n-1\).
\end{proof}

\begin{corollary}
	Let $\lambda = (\underbrace{0,\ldots,0}_j,\underbrace{1,\ldots,1}_k,\underbrace{0,\ldots,0}_l)$ with \(j, k, l \in \mathbb{Z}_+ \text{ and } j+k+l=n-1\) and let \(B = (1+\varepsilon)\cdot f^*_{\lambda}(s,t)\) for some \(\varepsilon > 0\). Then, the cardinality of the minimal complete set \(P^{NSP}_{min}(s,t)\) of \textsf{(UNSSPP(\(G,\lambda\)))} is intractable, i.e., the cardinality of the minimal complete set might be exponential in the number of vertices.
\end{corollary}
\begin{proof}
	Follows immediately from Theorem \ref{thm:intractable} and specifically as a special case of Corollary~\ref{cor:ksum}.
\end{proof}

As we have seen, for specific values of \(\lambda\) even a minimal complete set might be of exponential size. For these specific values it would be desirable to find a finite representation approximating/representing a minimal complete set of exponentially many universal near shortest simple paths from source to sink satisfying the given bound. This is not possible as we will show in the following. Consequently, we focus on \(\lambda = (1,\ldots,1)\) and $\lambda = (\underbrace{1,\ldots,1}_k,0,\ldots,0)$ with \(k\leq |V|-1\), since for these values even a minimal complete set might be of exponential size, see Theorem \ref{thm:intractable} and Corollary \ref{cor:ksum}.
\newline

Assume we are given \(\varepsilon\) and \(\delta\) with \(\varepsilon, \delta > 0\). Then, let \(U = f^*_{\lambda}(s,t)\) and \(B = (1+\varepsilon)\cdot f^*_{\lambda}(s,t)\). In order to obtain a finite representation, we divide the interval \([U,B]\) into polynomial many subintervals \([(1+\delta)^i\cdot U, (1+\delta)^{i+1}\cdot U)\) with \(i=0,1,\ldots,\left\lceil \frac{\log_2(\frac{B}{U})}{\log_2(1+\delta)}\right\rceil - 2\) and \([(1+\delta)^j\cdot U,B]\) for \(j = \left\lceil \frac{\log_2(\frac{B}{U})}{\log_2(1+\delta)}\right\rceil - 1\). Note that the range of \(i\) follows from the fact that \(U\cdot (1+\delta)^i\) has to be greater or equal to \(B\).
Now, we aim to find for every subinterval exactly one path, if one exists. Therefore, consider the following decision problem, called \textsf{FindPath}.
\newline

\textbf{\textsf{FindPath}\((G,\lambda,\varepsilon, \delta)\)}. Given a directed graph \(G=(V,A)\), \(\lambda \in \mathbb{Z}^{n-1}\), two distinct vertices \(s, t \in V\), values \(\varepsilon, \delta > 0\) and the universal optimal objective function value \(U=f^*_{\lambda}(s,t)\) corresponding to the universal optimal \(s\)-\(t\)-path and let \(B=(1+\varepsilon)\cdot f^*_{\lambda}(s,t)\). Decide whether there exists paths \(P^i \in \mathcal{P}_{st}\) such that \(U\cdot (1+\delta)^i \leq f_{\lambda}(P^i) < U\cdot (1+\delta)^{i+1}\) for all \(i=0,1,\ldots,\left\lceil \frac{\log_2(\frac{B}{U})}{\log_2(1+\delta)}\right\rceil - 2\) and a path \(P^j\) with \(U\cdot (1+\delta)^j \leq f_{\lambda}(P^j) \leq B\) for \(j = \left\lceil \frac{\log_2(\frac{B}{U})}{\log_2(1+\delta)}\right\rceil - 1\).

\begin{figure}[h!]
	\centering
	\begin{tikzpicture}
	\draw[ultra thick] (0,0) -- (13.5,0);
	\draw[ultra thick, -[] (0.595,0) -- (0.6,0);
	\node[] (Variables2) at (0.5,-0.5){\(U\)};
	
	\draw[ultra thick, -[] (11.195,0) -- (11.2,0);
	\draw[ultra thick] (11,0.2) to [bend left=45] (11,-0.2) node[anchor=north] {\((1+\delta)^j\cdot U\)};
	
	\draw[ultra thick, -[] (12.7,0) -- (12.695,0);
	\node[] (Variables2) at (12.7,-0.5){\(B\)};
	
	\draw[ultra thick] (2,0.2) to [bend left=45] (2,-0.2) node[anchor=north] {\((1+\delta)\cdot U\)};
	
	\draw[ultra thick, -[] (2.195,0) -- (2.2,0);
	
	\node[] (Variables) at (3.5,-0.5){\ldots};
	\draw[ultra thick] (5,0.2) to [bend left=45] (5,-0.2) node[anchor=north] {\((1+\delta)^i\cdot U\)};
	
	\draw[ultra thick, -[] (5.195,0) -- (5.2,0);
	
	\draw[ultra thick] (8,0.2) to [bend left=45] (8,-0.2) node[anchor=north] {\((1+\delta)^{i+1}\cdot U\)};
	
	\draw[ultra thick, -[] (8.195,0) -- (8.2,0);
	
	\node[] (Variables1) at (9.5,-0.5){\ldots};
	
	\end{tikzpicture}
	\begin{equation*}
	i=0,1,\ldots,\left\lceil \frac{\log_2(\frac{B}{U})}{\log_2(1+\delta)}\right\rceil - 2, \qquad j = \left\lceil \frac{\log_2(\frac{B}{U})}{\log_2(1+\delta)}\right\rceil -1
	\end{equation*}
	\caption{Illustration of \textsf{FindPath}}
\end{figure}

\begin{theorem}\label{thm:FindPath}
	For \(\lambda = (1,\ldots,1)\), the problem \textsf{FindPath} is \(\mathcal{NP}\)-complete, even for \(c(a) = 1\) for all \(a \in A\).
\end{theorem}
\begin{proof}
	\textsf{FindPath} is clearly in \(\mathcal{NP}\). To show that \textsf{FindPath} is \(\mathcal{NP}\)-complete, we conduct a polynomial time reduction from the directed Hamiltonian Path problem, which is known to be \(\mathcal{NP}\)-complete, cf. \citep{garey2002computers}.
	The reduction is as follows. Given an instance \(G=(V,A)\), a cost function \(c: A \rightarrow \mathbb{Z}_+\) with \(c(a)=1\) for all \(a \in A\) and two distinct vertices \(s, t \in V\) of the directed Hamiltonian Path problem (does there exists a hamiltonian path \(P \in \mathcal{P}_{st}\) in \(G\), i.e., \(c(P) = |V|-1\)), we construct an instance of \textsf{FindPath} as follows:
	Let \(G' = G\). Further, let \(c'(a) = c(a) = 1\) for all \(a \in A\). Moreover, let \(U\) denote the shortest path distance from \(s\) to \(t\) in \(G\), which can be computed in polynomial time using Dijkstra's algorithm. Further, we set \(\varepsilon=\frac{n-1}{U}-1\) and \(\delta=\frac{n-\frac{3}{2}}{U}-1\).
	Note that we obtain two subintervals \(S_1\) and \(S_2\) if we choose \(\varepsilon\) and \(\delta\) as described above, i.e., \(S_1=[U,n-\frac{3}{2})\) and \(S_2 = [n-\frac{3}{2},n-1]\).
	Now, there exists a path \(P\) in \(G'\) with \(c(P) \in S_2\) if and only if there exists a path \(P\) in \(G\) with \(c(P) = |V|-1\). This reduction can be done in polynomial time, which concludes the proof.
\end{proof}

\begin{remark}
	Note that in the proof of Theorem \ref{thm:FindPath}, we can easily find a path \(P\) with \(c(P) \in S_1\). Further,  we implicitly assumed that the shortest path distance in \(G\) is smaller than \(n-1\). This is not a restriction, since in unit-cost graphs where the shortest path distance is equal to \(n-1\), the Hamiltonian Path problem can be solved in polynomial time.
\end{remark}

\begin{corollary}
	For $\lambda = (\underbrace{1,\ldots,1}_k,0,\ldots,0)$ with \(k\leq n -1\), the problem \textsf{FindPath} is \(\mathcal{NP}\)-complete, even for \(c(a) = 1\) for all \(a \in A\).
\end{corollary}
\begin{proof}
	Follows immediately from Theorem \ref{thm:FindPath} for \(k := n-1\).
\end{proof}

Consequently, there is no polynomial time algorithm that finds a finite representation of the form as described above, unless \(\mathcal{P}=\mathcal{NP}\).

\section{Conclusion}\label{sec:conclusion}
In this paper, we proposed a generalization of the classical near shortest simple paths problem, called the universal near shortest simple paths problem, by introducing a universal weight vector \(\lambda\). We showed that this problem is intractable, i.e., the number of universal near shortest simple paths might be exponential in the number of vertices. We presented two different algorithms, for which we showed that the amount of work per path enumerated is polynomially bounded as long as the underlying universal shortest path problem with respect to \(\lambda\) can be solved in polynomial time. Our fastest algorithm (when applied to the classical near shortest paths problem) has the same running time complexity per path enumerated than the best known algorithm for the near shortest path problem with the addition that it can be applied to almost any shortest path problem.  
Further, we showed how to generate a minimal complete set of (universal) near shortest simple paths with respect to \(\lambda\) and proved the worst-case size of this set. In particular, we have seen that there are values of \(\lambda\), where even a minimal complete set might be of exponential size. For these values we proved that finding a finite selection of paths representing the possible exponentially large set of universal near shortest simple paths is still a hard task and cannot be accomplished in polynomial time.

\section*{Acknowledgments}
	This work was partially supported by the Bundesministerium für Bildung und Forschung (BMBF) under Grant No. 13N14561.





\bibliographystyle{apalike}

%
%
%
\end{document}